\documentclass[twoside,leqno,10pt, A4]{amsart}
\usepackage{amsfonts}
\usepackage{amsmath}
\usepackage{amscd}
\usepackage{amssymb}
\usepackage{amsthm}
\usepackage{amsrefs}
\usepackage{latexsym}
\usepackage{mathrsfs}
\usepackage{bbm}
\usepackage{amscd}
\usepackage{amssymb}
\usepackage{amsthm}
\usepackage{amsrefs}
\usepackage{latexsym}
\usepackage{mathrsfs}
\usepackage{bbm}
\usepackage{enumerate}
\usepackage{graphicx}
\usepackage{color}
\begin{document}

\newtheorem{theorem}[subsection]{Theorem}
\newtheorem{proposition}[subsection]{Proposition}
\newtheorem{lemma}[subsection]{Lemma}
\newtheorem{corollary}[subsection]{Corollary}
\newtheorem{conjecture}[subsection]{Conjecture}
\newtheorem{prop}[subsection]{Proposition}
\newtheorem{defin}[subsection]{Definition}

\numberwithin{equation}{section}
\newcommand{\mr}{\ensuremath{\mathbb R}}
\newcommand{\mc}{\ensuremath{\mathbb C}}
\newcommand{\dif}{\mathrm{d}}
\newcommand{\intz}{\mathbb{Z}}
\newcommand{\ratq}{\mathbb{Q}}
\newcommand{\natn}{\mathbb{N}}
\newcommand{\comc}{\mathbb{C}}
\newcommand{\rear}{\mathbb{R}}
\newcommand{\prip}{\mathbb{P}}
\newcommand{\uph}{\mathbb{H}}
\newcommand{\fief}{\mathbb{F}}
\newcommand{\majorarc}{\mathfrak{M}}
\newcommand{\minorarc}{\mathfrak{m}}
\newcommand{\sings}{\mathfrak{S}}
\newcommand{\fA}{\ensuremath{\mathfrak A}}
\newcommand{\mn}{\ensuremath{\mathbb N}}
\newcommand{\mq}{\ensuremath{\mathbb Q}}
\newcommand{\half}{\tfrac{1}{2}}
\newcommand{\f}{f\times \chi}
\newcommand{\summ}{\mathop{{\sum}^{\star}}}
\newcommand{\chiq}{\chi \bmod q}
\newcommand{\chidb}{\chi \bmod db}
\newcommand{\chid}{\chi \bmod d}
\newcommand{\sym}{\text{sym}^2}
\newcommand{\hhalf}{\tfrac{1}{2}}
\newcommand{\sumstar}{\sideset{}{^*}\sum}
\newcommand{\sumprime}{\sideset{}{'}\sum}
\newcommand{\sumprimeprime}{\sideset{}{''}\sum}
\newcommand{\sumflat}{\sideset{}{^\flat}\sum}
\newcommand{\shortmod}{\ensuremath{\negthickspace \negthickspace \negthickspace \pmod}}
\newcommand{\V}{V\left(\frac{nm}{q^2}\right)}
\newcommand{\sumi}{\mathop{{\sum}^{\dagger}}}
\newcommand{\mz}{\ensuremath{\mathbb Z}}
\newcommand{\leg}[2]{\left(\tfrac{#1}{#2}\right)}
\newcommand{\muK}{\mu_{\omega}}
\newcommand{\thalf}{\tfrac12}
\newcommand{\lp}{\left(}
\newcommand{\rp}{\right)}
\newcommand{\Lam}{\Lambda_{[i]}}
\newcommand{\lam}{\lambda}
\newcommand{\af}{\mathfrak{a}}
\newcommand{\sw}{S_{[i]}(X,Y;\Phi,\Psi)}
\newcommand{\lz}{\left(}
\newcommand{\pz}{\right)}
\newcommand{\bfrac}[2]{\lz\frac{#1}{#2}\pz}
\newcommand{\odd}{\mathrm{\ primary}}
\newcommand{\even}{\text{ even}}
\newcommand{\res}{\mathrm{Res}}
\newcommand{\sumn}{\sumstar_{(c,1+i)=1}  w\left( \frac {N(c)}X \right)}
\newcommand{\lab}{\left|}
\newcommand{\rab}{\right|}
\newcommand{\Go}{\Gamma_{o}}
\newcommand{\Ge}{\Gamma_{e}}
\newcommand{\M}{\widehat}
\newcommand{\ch}{\mathrm{conv}}
\newcommand{\A}{\mathfrak{a}}

\theoremstyle{plain}
\newtheorem{conj}{Conjecture}
\newtheorem{remark}[subsection]{Remark}

\makeatletter
\def\widebreve{\mathpalette\wide@breve}
\def\wide@breve#1#2{\sbox\z@{$#1#2$}%
     \mathop{\vbox{\m@th\ialign{##\crcr
\kern0.08em\brevefill#1{0.8\wd\z@}\crcr\noalign{\nointerlineskip}%
                    $\hss#1#2\hss$\crcr}}}\limits}
\def\brevefill#1#2{$\m@th\sbox\tw@{$#1($}%
  \hss\resizebox{#2}{\wd\tw@}{\rotatebox[origin=c]{90}{\upshape(}}\hss$}
\makeatletter

\title[Non-vanishing of central values of angular Hecke $L$-functions]{Non-vanishing of central values of angular Hecke $L$-functions}

\author[P. Gao]{Peng Gao}
\address{School of Mathematical Sciences, Beihang University, Beijing 100191, China}
\email{penggao@buaa.edu.cn}

\begin{abstract}
In this paper, we evaluate asymptotically the first and second moment of central values of the family of Hecke $L$-functions associated to angular characters over imaginary quadratic number fields of class number one. We use our results to deduce a non-vanishing result concerning these $L$-functions. 
\end{abstract}

\maketitle



\section{Introduction}\label{sec 1}

 Let  $K=\mq(\sqrt{d})$ be an imaginary quadratic number field of class number one. Here $d$ is a negative, square-free integer and it is well-known that
 (see  \cite[(22.77)]{iwakow}) the class number of $K$ is one if and only if
\begin{align}
\label{dvalue}
 d \in \{-1, -2, -3, -7,-11,-19,-43,-67,-163 \}.
\end{align}
  We denote $\mathcal O_K, U_K, D_K$ the ring of integers, the group of units and the discriminant of $K$, respectively.
 Note that (see \cite[Section 2.1]{Holm25}) the order $|U_K|$ of $U_K$ equals $4$ when $d=-1$, equals $6$ when $d=-3$ and equals $2$ for other values of $d$ listed
 in \eqref{dvalue}.  Throughout the paper,
we denote $M_K$ the integer such that $2M_K=|U_K|$.

  Denote $\A$ an integral ideal of $\mathcal O_K$. 
  As $\mathcal O_K$ is a principal ideal domain, the generators of $\A=(\alpha ) \subseteq \mathcal O_K$ differ
  by multiplication of an element in  $U_K$. In what follows we shall therefore fix a generator $\alpha$ of $(\alpha )$  such that
  $\theta_{\A} := \theta_{\alpha}$ lies in $[-\pi /|U_K|, \pi/|U_K|)$,  where $\theta_{\alpha}$ denotes the argument of $\alpha$.

  For any $l \in \mz$, we define the Hecke character
\begin{equation*}
\lambda_{M_Kl}(\A)= \left(\frac {\alpha}{ \overline{\alpha}}\right)^{M_Kl} = e^{2i M_Kl \theta_{\A}}.
\end{equation*}
  Note that as $\lambda_{M_Kl}$ is trivial on $U_K$, it defines a Hecke character of trivial infinite type. We write the associated Hecke $L$-function as
\begin{align*}
L(s, \lambda_{M_Kl}) &= \ \sum_{\substack{\A \subseteq \mathcal O_K\\ \A \neq 0}}\frac{\lambda_{M_Kl}(\A)}{N(\A)^{s}}=\prod_{(\varpi) }\left(1-\frac{\lambda_{M_Kl}(\varpi)}{N(\varpi)^s}\right)^{-1}, \hspace{5mm}
\Re(s)>1,
\end{align*}
  where $N(\A)$ denotes the norm of the integral ideal $\A$, and we reserve the letter $\varpi$ for a prime in $\mathcal O_K$ throughout the paper,
  by which we mean that the ideal $(\varpi)$ generated by it is a prime ideal.

  Let $X>1$ be a large real number, interesting properties concerning
   the family of Hecke $L$-functions $\{ L(s, \lambda_{M_Kl}) : 1 \leq l \leq X \}$ have been investigated in \cite{Waxman21}, \cite{Holm25} for its one-level density,
   in \cite{white22} for its extreme values.
   
   Under the generalized Riemann hypothesis, the one-level density results established in \cite{Waxman21}, \cite{Holm25} (see \cite[Corollary 1.3]{Waxman21}) imply that at least $75\%$ of the members of the family of angular Hecke $L$-functions do not vanish at the central point.
   
    It is the aim of this paper to establish a non-vanishing result unconditionally for these families of $L$-functions by evaluating their first and second moments at the central point. 

Our first result evaluates asymptotically the first moment as follows.
\begin{theorem}
\label{1stmoment}
 With the notation as above. Let $K$ be any imaginary quadratic number field of class number  one. Let $M_K$ be the integer such that $2M_K=|U_K|$, we have for large $X>1$,
\begin{align}
\label{1stmomentasymp}
 \sum_{X < l \leq 2X}L(\frac 12, \lambda_{M_Kl})=X\log X+(\log (4M_K))X+O(X^{1/2}\log X).
\end{align}
\end{theorem}

   Our next result evaluates asymptotically the second moment.
\begin{theorem}
\label{2ndmoment}
 With the notation as above. Let $K$ be any imaginary quadratic number field of class number one. Let $M_K$ be the integer such that $2M_K=|U_K|$, we have for large $X>1$,
\begin{align}
\label{2stmomentasymp}
 \sum_{X < l \leq 2X}\Big |L(\frac 12, \lambda_{M_Kl})\Big |^2=C_KX(\log X)^3+O(X(\log X)^2),
\end{align}
 where $C_K$ is a constant depending on $K$ only.
\end{theorem}

   It follows from Theorem \ref{1stmoment}-\ref{2ndmoment} and the standard arguments using ‌Cauchy-Schwarz inequality that we have the following non-vanishing result concerning the family of angular Hecke $L$-functions.
\begin{theorem}
\label{coro:nonvanish}
With the notation as above.  Let $K$ be any imaginary quadratic number field of class number  one and let $M_K$ be the integer such that $2M_K=|U_K|$. There exist infinitely many  angular Hecke  characters $\lambda_{M_Kl}$ such that $L(1/2, \lambda_{M_Kl})$ is non-zero.  More precisely, the number of such characters with $X< l \leq 2X$ is $\gg X/\log X$ for any large $X$.
\end{theorem}

\section{Preliminaries}
\label{sec 2}

\subsection{Angular Hecke characters}

  It is shown in \cite[Section 2]{Waxman21} that for $l \neq 0$, the complete Hecke $L$-function
\begin{equation}
\label{LambdaL}
\Lambda_{l}(s) = |D_K|^{s/2}(2\pi)^{-s}\Gamma(s+|M_Kl|)L(s, \lambda_{M_Kl})
\end{equation}
  is entire and have an analytic continuation to the whole complex plane that satisfies the functional equation
\begin{equation}
\label{fcneqn}
 \Lambda_{l}(s) = \Lambda_{l}(1-s).
\end{equation}

  The above allows us to establish the following approximate functional equation for $L(\frac 12, \lambda_{M_Kl})$.
\begin{lemma}
\label{approximate functional equation}
 With the notation as above. We have for any $l \neq  0$,
\begin{align}
\label{afe}
L(\frac 12, \lambda_{M_Kl})
= 2\sum_{\substack{\mathfrak a\subseteq\mathcal O_K\\ \mathfrak a\ne(0)}}
\frac{\lambda_{M_Kl}(\mathfrak a)}{N(\mathfrak a)^{1/2}}
V_{|M_Kl|}\left(\frac {2\pi N(\mathfrak a)}{\sqrt{|D_K|}}\right),
\end{align}
 where
\begin{align}
\label{Vk}
V_{x}(y)
=
\frac{1}{2\pi i}
\int_{(3)}
\frac{\Gamma(\frac{1}{2}+x+s)}{\Gamma(\frac{1}{2}+x)}
y^{-s}e^{s^2}\frac{ds}{s}.
\end{align}
\end{lemma}
\begin{proof}
  We apply \cite[Theorem~5.3]{iwakow} by setting $G(s)=e^{s^2}$ there and use \eqref{LambdaL} and \eqref{fcneqn} to see that we have
\[
L(\frac 12, \lambda_{M_Kl})
= \sum_{\substack{\mathfrak a\subseteq\mathcal O_K\\ \mathfrak a\ne(0)}}
\frac{\lambda_{M_Kl}(\mathfrak a)+\overline{\lambda_{M_Kl}(\mathfrak a)}}{N(\mathfrak a)^{1/2}}
V_{|M_Kl|}\left(N(\mathfrak a)\right),
\]
  where $V_{|M_Kl|}$ is given in \eqref{Vk}. We now set $N(\mathfrak a)=n$ to see that we may recast the above as
\[
L(\frac 12, \lambda_{M_Kl})
= \sum_{\substack{n \geq 1}}
\frac{A_{M_Kl}(n)+\overline{A_{M_Kl}(n)}}{n^{1/2}}
V_{|M_Kl|}\left(n\right),
\]
  where
\begin{align*}
A_{M_Kl}(n) &= \sum_{\substack{N(\mathfrak a) = n}}\lambda_{M_Kl}(\mathfrak a).
\end{align*}
  Note that it is shown in \cite[p. 251]{Holm25} that we have $\overline{A_{M_Kl}(n)}=A_{M_Kl}(n)$ for all $A_{M_Kl}(n)$. It follows that
\begin{align}
\label{Lexp}
L(\frac 12, \lambda_{M_Kl})
= 2\sum_{\substack{n \geq 1}}
\frac{A_{M_Kl}(n)}{n^{1/2}}
V_{|M_Kl|}\left(n\right).
\end{align}
 We reverse the above process to see that the relation given in \eqref{afe} is valid. This completes the proof of the lemma.
\end{proof}

  We note that similar to \cite[Lemma 2.1]{sound1}, one shows that the function $V_{|M_Kl|}\left(x \right)$ is real-valued for any real $x>0$. We deduce
from this and \eqref{Lexp} by recalling that  $\overline{A_{M_Kl}(n)}=A_{M_Kl}(n)$ there to see that we have $L(\frac 12, \lambda_{M_Kl}) \in \mr$. It follows from this
that we have
\begin{align}
\label{Lsquare}
|L(\frac 12, \lambda_{M_Kl})|^2=L(\frac 12, \lambda_{M_Kl})^2.
\end{align}

  Our next result establishes the approximate functional equation for $|L(\frac 12, \lambda_{M_Kl})|^2$.
\begin{lemma}
\label{approximate functional equation-square}
 With the notation as above. We have for any $l \neq  0$,
\begin{align}
\label{afesquare}
	\Big|L\left(\frac{1}{2},\lambda_{M_Kl}\right)\Big|^2
	=
	2 \sum_{\substack{\mathfrak n \subseteq\mathcal O_K\\ \mathfrak n\ne(0)}}\frac{\tau_K(\mathfrak n)\lambda_{M_Kl}(\mathfrak n)}{N(\mathfrak n)^{1/2}}
	W_{|M_Kl|}\left(\frac {(2\pi)^2N(\mathfrak n)}{|D_K|}\right),
\end{align}
 where
\[
\tau_K(\mathfrak n)
:=
\sum_{\substack{\mathfrak a\subseteq\mathcal O_K\\
\mathfrak a \mid \mathfrak n}}
1,
\]
 and where
\begin{align}
\label{Wk}
W_{x}(y)
=
\frac{1}{2\pi i}
\int_{(3)}
\frac{\Gamma(\frac{1}{2}+x+s)^2}{\Gamma(\frac{1}{2}+x)^2}
y^{-s}e^{s^2}\frac{ds}{s}.
\end{align}
\end{lemma}
\begin{proof}
  Consider the integral
	\[
		I
		=
		\frac{1}{2\pi i}
		\int_{(3)}
		\Lambda^2_{M_Kl}(\tfrac{1}{2} + s)e^{s^2}
  \frac{\mathrm{d}s}{s}.
	\]
    Then by Cauchy's theorem, we have
\begin{align}
\label{lambdaI}
    	\Lambda^2_{M_Kl}(\frac{1}{2})
	=2I.
\end{align}
Note that by \eqref{LambdaL} and \eqref{Lsquare},
\begin{align}
\label{lambdasquarecentral}
\Lambda^2_{M_Kl}(\frac{1}{2})
=
\frac{|D_K|^{1/2}\Gamma(\frac 12+|M_Kl|)^2}{2\pi}
|L(\frac12,\lambda_{M_Kl})|^2. 
\end{align}
  On the other hand, we may expand \(\Lambda^2_{M_Kl}(\tfrac{1}{2} + s)\) into Dirichlet series and grouping the ideals according to
\(\mathfrak n=N(\mathfrak a)N(\mathfrak b)\) to evaluate $I$. We deduce from this and \eqref{lambdaI}--\eqref{lambdasquarecentral} that the expression given in 
\eqref{afesquare} is valid. This completes the proof of the lemma.
\end{proof}

\section{Proof of Theorem \ref{1stmoment}}
\label{sec:proofofmainthm}

     We apply Lemma \ref{approximate functional equation} and \eqref{Vk} to see that
\begin{align}
\label{Ldcomp}
  \sum_{X < l \leq 2X}L(\frac 12, \lambda_{M_Kl})
= 2\sum_{X < l \leq 2X}\sum_{\substack{\mathfrak a\subseteq\mathcal O_K\\ \mathfrak a\ne(0)}}
\frac{\lambda_{M_Kl}(\mathfrak a)}{N(\mathfrak a)^{1/2}}
V_{M_Kl}\left(\frac {2\pi N(\mathfrak a)}{\sqrt{|D_K|}}\right):=S+R_1+R_2,
\end{align}
  where
\begin{align*}
  S
= & 2 \sum_{X < l \leq 2X}\sum_{\substack{\mathfrak a\subseteq\mathcal O_K\\ \mathfrak a=(n) \\ n \in \mn}}
\frac{\lambda_{M_Kl}(\mathfrak a)}{N(\mathfrak a)^{1/2}}
V_{M_Kl}\left(\frac {2\pi N(\mathfrak a)}{\sqrt{|D_K|}}\right),  \\
R_1
=& 2 \sum_{X < l \leq 2X}\sum_{\substack{\mathfrak a\subseteq\mathcal O_K, N(\mathfrak a) \leq X \\ \mathfrak a\ne(0) \\ \mathfrak a \neq (n), n \in \mn}}
\frac{\lambda_{M_Kl}(\mathfrak a)}{N(\mathfrak a)^{1/2}}
V_{M_Kl}\left(\frac {2\pi N(\mathfrak a)}{\sqrt{|D_K|}}\right), \\
R_2
= & 2 \sum_{X < l \leq 2X}\sum_{\substack{\mathfrak a\subseteq\mathcal O_K, N(\mathfrak a) > X \\ \mathfrak a\ne(0) \\ \mathfrak a \neq (n), n \in \mn}}
\frac{\lambda_{M_Kl}(\mathfrak a)}{N(\mathfrak a)^{1/2}}
V_{M_Kl}\left(\frac {2\pi N(\mathfrak a)}{\sqrt{|D_K|}}\right).
\end{align*}

   Note that when $\mathfrak a=(n)$ with $n \in \mn$, we have $\lambda_{M_Kl}(\mathfrak a)=1$ for all $X < l \leq 2X$. It follows that we have
\begin{align*}
  S
= 2\sum_{X < l \leq 2X}\sum_{\substack{n \in \mn}}
\frac{1}{n}
V_{M_Kl}\left(\frac {2\pi n^2}{\sqrt{|D_K|}}\right)=\sum_{X < l \leq 2X}\frac 2{2\pi i}\int_{(3)}
\frac{\Gamma(\frac{1}{2}+M_Kl+s)}{\Gamma(\frac{1}{2}+M_Kl)}
\left(\frac {2\pi }{\sqrt{|D_K|}}\right)^{-s}e^{s^2}\zeta(1+2s)\frac{ds}{s},
\end{align*}
  where $\zeta(s)$ is the Riemann zeta function. We now shift the line of integration above to $\Re(s)=-1$. As $\zeta(s)$ has a simple pole
at $s=1$ with the corresponding residue being $1$, we see that we encounter a double pole at $s=0$ in the process, whose residue can be easily computed to be
\begin{align*}
 \frac{\Gamma'(\frac{1}{2}+M_Kl)}{2\Gamma(\frac{1}{2}+M_Kl)}+O(1),
\end{align*}
  where the implied constant depends on $K$ only. It follows that
\begin{align}
\label{Sest}
  S= \sum_{X < l \leq 2X}\frac{\Gamma'(\frac{1}{2}+M_Kl)}{\Gamma(\frac{1}{2}+M_Kl)}+O(1)+S_1,
\end{align}
  where
\begin{align*}
  S_1= \sum_{X < l \leq 2X}\frac 2{2\pi i}\int_{(-1)}
\frac{\Gamma(\frac{1}{2}+M_Kl+s)}{\Gamma(\frac{1}{2}+M_Kl)}
\left(\frac {2\pi }{\sqrt{|D_K|}}\right)^{-s}e^{s^2}\zeta(1+2s)\frac{ds}{s}.
\end{align*}

   Note that by Stirling's formula (see \cite[(5.112)]{iwakow}), we have
\begin{align*}
  \log \Gamma(s)=\frac12\log (2\pi)+(s-\frac 12)\log s-s+O(\frac 1{|s|}).
\end{align*}
   It follows from this that for $|\sigma| \leq 10$ and $X < u \leq 2X$,  we have
\begin{align*}
  \log \Gamma(\frac{1}{2}+M_Ku+\sigma)-\log \Gamma(\frac{1}{2}+M_Ku) =&(M_Ku+\sigma)\log (\frac{1}{2}+M_Ku+\sigma)-M_Ku\log (\frac{1}{2}+M_Ku)-\sigma+
  O(\frac 1{X}) \\
  =&(M_Ku+\sigma)\log (\frac{1}{2}+M_Ku)-M_Ku\log (\frac{1}{2}+M_Ku)+O(1) \\
  =& \sigma\log (\frac{1}{2}+M_Ku)+O(1).
\end{align*}
  The above implies that $|\sigma| \leq 10$ and $X < u \leq 2X$,  we have
\begin{align}
\label{Gammaratiobound}
  \frac{\Gamma(\frac{1}{2}+M_Ku+\sigma)}{\Gamma(\frac{1}{2}+M_Ku)}
  \ll & (\frac{1}{2}+M_Ku)^\sigma \ll X^{\sigma}.
\end{align}

  Note also that we have for $x \geq 1, y \in \mr$,
\begin{align}
\label{Gammabound}
 |\Gamma(x+iy)| \leq \Gamma(x).
\end{align}
 We apply \eqref{Gammaratiobound} and \eqref{Gammabound} to see that
\begin{align*}
  \frac {|\Gamma(-\frac 12+M_Kl+it)|}{\Gamma(\frac 12+M_Kl)}\ll \frac{\Gamma(-\frac 12+M_Kl)}{\Gamma(\frac 12+M_Kl)} \ll \frac {1}{-\frac 12+M_Kl}.
\end{align*}

  It follows that
\begin{align}
\label{S1est}
  S_1 \ll \int_{(-1)}\Big |e^{s^2}\zeta(1+2s)\frac{ds}{s}\Big | \ll \int_{(-1)}\Big |e^{s^2}\zeta(-2s)\frac {\Gamma(-2s)}{\Gamma(1+2s)}\frac{ds}{s}\Big |,
\end{align}
  where the last estimation above follows from the well-known functional equation (see \cite[Corollary 10.3]{MVa1}) for $\zeta(s)$ that
\begin{align*}
 \zeta(s)\Gamma(s)\pi^{-s/2}=\zeta(1-s)\Gamma(1-s)\pi^{-(1-s)/2}.
\end{align*}

  Stirling's formula (see \cite[(5.113)]{iwakow}) yields
\begin{align*}
  \frac {\Gamma(-2s)}{\Gamma(1+2s)} \ll (1+|s|)^{-4\Re (s)-1}.
\end{align*}
  We deduce from this and \eqref{S1est} that
\begin{align*}
  S_1 \ll 1.
\end{align*}

    We apply the above into \eqref{Sest} to see that
\begin{align}
\label{Msum}
  S= \sum_{X < l \leq 2X}\frac{\Gamma'(\frac{1}{2}+M_Kl)}{\Gamma(\frac{1}{2}+M_Kl)}+O(1).
\end{align}
   We further note from \cite[(C.17)]{MVa1} that we have
\begin{align*}
  \frac{\Gamma'(\frac{1}{2}+M_Kl)}{\Gamma(\frac{1}{2}+M_Kl)}=\log (\frac{1}{2}+M_Kl)+O(\frac 1{1/2+M_Kl})=\log (M_Kl)+O(1).
\end{align*}
   It follows from this and \eqref{Msum} that we have by partial summation,
\begin{align}
\label{Mval}
  S= \sum_{X < l \leq 2X}\log (M_Kl)+O(1)=X\log X+(\log (4M_K))X+O(\log X).
\end{align}

  Next, we estimate $R_1$ and $R_2$ by keeping the line of integration in $V_{M_Kl}$ to be at $\Re(s)=3$ for $R_2$ while shifting the line of integration
   to $\Re(s)=-1$ for $R_1$. In this process we encounter a simple pole with residue $1$ at $s=0$, whose contribution to $R_1$ is denoted by $R_{1,1}$. We denote $R_{1,2}$ the contribution to $R_1$ on the new line of integration, so that we have
\begin{align}
\label{R1decomp}
  R_1=R_{1,1}+R_{1,2},
\end{align}
   where
\begin{align}
\label{R112exp}
\begin{split}
R_{1,1}=& 2 \sum_{X < l \leq 2X}\sum_{\substack{\mathfrak a\subseteq\mathcal O_K, N(\mathfrak a) \leq X \\ \mathfrak a\ne(0) \\ \mathfrak a \neq (n), n \in \mn}}
\frac{\lambda_{M_Kl}(\mathfrak a)}{N(\mathfrak a)^{1/2}}, \\
R_{1,2}
= & 2 \sum_{X < l \leq 2X}\sum_{\substack{\mathfrak a\subseteq\mathcal O_K, N(\mathfrak a) \leq X \\ \mathfrak a\ne(0) \\ \mathfrak a \neq (n), n \in \mn}}
\frac{\lambda_{M_Kl}(\mathfrak a)}{N(\mathfrak a)^{1/2}}\cdot \frac{1}{2\pi i}
\int_{(-1)}
\frac{\Gamma(\frac{1}{2}+M_Kl+s)}{\Gamma(\frac{1}{2}+M_Kl)}
\left(\frac {2\pi N(\mathfrak a)}{\sqrt{|D_K|}}\right)^{-s}e^{s^2}\frac{ds}{s}.
\end{split}
\end{align}

   We note that when $\mathfrak a \neq (n)$ with $n \in \mn$, we have $\theta_{\A} \notin  \pi \mz$, so that similar to \cite[(6.3)]{Waxman21}, we have for any real
$u \geq 1$,
\begin{align}
\label{angle average}
 \sum_{1 \leq l\leq u}\lambda_{M_Kl}(\A) \ll \frac{1}{|\sin (M_K\theta_{\A})|}.
\end{align}
  Now, our choice of $\theta_K \in [-\pi/|U_K|, \pi/|U_K|)$ and the value of $M_K=|U_K|/2$ ensures that we have
  $M_K\theta_{\A} \in [-\pi/2, \pi/2)$. It then follows from the well-known estimation $ 2x/\pi \leq \sin x \leq x$ for $0<x \leq \pi/2$ that we have
\begin{align}
\label{sbound}
 \frac{1}{|\sin (M_K\theta_{\A})|} \ll \frac{1}{|\theta_{\A}|}\ll \frac{1}{|\sin (\theta_{\A})|}.
\end{align}
  Recall (see \cite[Corollary 2, Chap. 2]{Marcus18}) that the ring of integers $\mathcal{O}_K$ is a free $\mz$ module such that $\mathcal{O}_K=\mz+\omega_K \mz$, where
\begin{align*}
   \omega_K & =\begin{cases}
     \frac {1}{2}(1+\sqrt{d}) \qquad & \text{if $d \equiv 1 \pmod 4$ }, \\
     \sqrt{d} \qquad & \text{if $d \equiv 2, 3 \pmod 4$ }.
    \end{cases}
\end{align*}
   We apply this to see that we may write any $\mathfrak a \neq (n)$ with $n \in \mn$ as $\mathfrak a=(a+b\omega_K)$ with $a, b \in \mz, b \neq 0$. It follows from this that we have
\begin{align*}
 |\sin (\theta_{\A})| \gg \frac {|b|}{\sqrt{N(\A)}},
\end{align*}
   so that
\begin{align}
\label{thetabound}
 \frac 1{|\sin (\theta_{\A})|} \ll \frac {\sqrt{N(\A)}}{|b|}.
\end{align}
   We deduce from \eqref{angle average}-\eqref{thetabound} that for any $\mathfrak a=(a+b\omega_K)$ with $a, b \in \mz, b \neq 0$ and any $X \leq u \leq 2X$,
\begin{align}
\label{lambdaps}
 \sum_{X < l\leq u}\lambda_{M_Kl}(\A) \ll \frac {\sqrt{N(\A)}}{|b|}.
\end{align}

   We note moreover that for any $u \geq 5$,
\begin{align}
\label{bpartialsum}
  \sum_{\substack{\mathfrak a=(a+b\omega_K), b\neq 0, N(\mathfrak a) \leq u }} \frac{1}{|b|} \ll \sum_{1 \leq b \leq \sqrt{u}}\frac 1{b} \sum_{1 \leq a \leq \sqrt{u}}1 \ll u^{1/2}\log u.
\end{align}
  In particular, this together with \eqref{lambdaps} implies that
\begin{align}
\label{R11est}
  R_{1,1} \ll X^{1/2}\log X.
\end{align}

  We further deduce from \eqref{lambdaps} and partial summation that
  when $\mathfrak a=(a+b\omega_K)$ with $a, b \in \mz, b \neq 0$, we have
\begin{align}
\label{sumlambda}
 \sum_{X < l\leq 2X}\frac {\lambda_{M_Kl}(\A)}{\sqrt{N(\A)}} \frac{\Gamma(\frac{1}{2}+M_Kl+s)}{\Gamma(\frac{1}{2}+M_Kl)}\ll
 \frac 1{|b|}\Big |\frac{\Gamma(\frac{1}{2}+M_KX+s)}{\Gamma(\frac{1}{2}+M_KX)}\Big |+\frac 1{|b|}
 \Big |\frac{\Gamma(\frac{1}{2}+2M_KX+s)}{\Gamma(\frac{1}{2}+2M_KX)}\Big | + \frac 1{|b|}\int^{2X}_X\Big |h'(u)\Big |du,
\end{align}
  where we set
\begin{align*}
 h(u):=\frac{\Gamma(\frac{1}{2}+M_Ku+s)}{\Gamma(\frac{1}{2}+M_Ku)}.
\end{align*}

  We apply \cite[Lemma 3]{white22} to see that when $\mathfrak a=(a+b\omega_K)$ with $a, b \in \mz, b \neq 0$,
\begin{align}
\label{rhoderbound}
 |h'(u)| \ll |h(u)|\cdot |\log (1+\frac {s}{M_Ku})| \ll |h(u)|\cdot \log (1+|\frac {s}{M_Ku}|) \ll |h(u)|\cdot |\frac {s}{u}|,
\end{align}
  where the last estimation above follows from the well-known bound that $1+x \leq e^x$ for any $x \in \mr$.

  We substitute \eqref{rhoderbound} into \eqref{sumlambda} to deduce that for $\mathfrak a=(a+b\omega_K)$ with $a, b \in \mz, b \neq 0$,
\begin{align*}
 \sum_{X < l\leq 2X}\frac {\lambda_{M_Kl}(\A)}{\sqrt{N(\A)}} \frac{\Gamma(\frac{1}{2}+M_Kl+s)}{\Gamma(\frac{1}{2}+M_Kl)}\ll
 \frac {|s|}{|b|}\max_{X < u \leq 2X} \Big ( \Big |\frac{\Gamma(\frac{1}{2}+M_Ku+s)}{\Gamma(\frac{1}{2}+M_Ku)}\Big |\Big ).
\end{align*}

   Using the relation given in \eqref{Gammabound}, we infer from the above that for $\mathfrak a=(a+b\omega_K)$ with $a, b \in \mz, b \neq 0$,
\begin{align}
\label{lambdapartialsum}
\begin{split}
 \sum_{X < l\leq 2X}\frac {\lambda_{M_Kl}(\A)}{\sqrt{N(\A)}} \frac{\Gamma(\frac{1}{2}+M_Kl+s)}{\Gamma(\frac{1}{2}+M_Kl)}\ll
 \frac {|s|}{|b|}\max_{X < u \leq 2X} \Big ( \Big |\frac{\Gamma(\frac{1}{2}+M_Ku+\sigma)}{\Gamma(\frac{1}{2}+M_Ku)}\Big |\Big ),
\end{split}
\end{align}
 where we set $s=\sigma+it$ with $\sigma, t \in \mr$.
  
   We apply \eqref{Gammaratiobound} into \eqref{lambdapartialsum} to see that for $|\sigma| \leq 10$,  we have for $\mathfrak a=(a+b\omega_K)$ with $a, b \in \mz, b \neq 0$,
\begin{align}
\label{sumlambdagamma}
\begin{split}
 \sum_{X < l\leq 2X}\frac {\lambda_{M_Kl}(\A)}{\sqrt{N(\A)}} \frac{\Gamma(\frac{1}{2}+M_Kl+s)}{\Gamma(\frac{1}{2}+M_Kl)}\ll
  \frac {|s|}{|b|}X^{\sigma}.
\end{split}
\end{align}

  It follows from the above and \eqref{R112exp} that we have 
\begin{align}
\label{R12bound}
  R_{1,2} \ll \sum_{\substack{\mathfrak a=(a+b\omega_K), b\neq 0, N(\mathfrak a) \leq X }} \frac{N(\mathfrak a)}{|b|X} \ll X^{1/2}\log X,
\end{align}
  where the last estimation above follows from \eqref{bpartialsum} and partial summation.

   Similarly, we have
\begin{align}
\label{R3bound}
  R_2 \ll & X^{1/2}\log X.
\end{align}

   The asymptotical expression given in \eqref{1stmomentasymp} now follows from \eqref{Ldcomp}, \eqref{Mval}, \eqref{R1decomp}, \eqref{R11est}, \eqref{R12bound} and \eqref{R3bound}. This completes the proof of Theorem \ref{1stmoment}.

\section{Proof of Theorem \ref{2ndmoment}}
\label{sec:proofof2ndmainthm}

     We apply Lemma \ref{approximate functional equation-square} and \eqref{Wk} to see that
\begin{align}
\label{L2dcomp}
  \sum_{X < l \leq 2X}|L(\frac 12, \lambda_{M_Kl})|^2
= 2\sum_{X < l \leq 2X}\sum_{\substack{\mathfrak n \subseteq\mathcal O_K\\ \mathfrak n \ne(0)}}
\frac{\tau_K(\mathfrak n)\lambda_{M_Kl}(\mathfrak n)}{N(\mathfrak n)^{1/2}}
W_{M_Kl}\left(\frac {(2\pi)^2 N(\mathfrak n)}{|D_K|}\right):=M+T_1+T_2,
\end{align}
  where
\begin{align*}
  M
= & 2 \sum_{X < l \leq 2X}\sum_{\substack{\mathfrak n \subseteq\mathcal O_K\\ \mathfrak n=(n) \\ n \in \mn}}
\frac{\tau_K(\mathfrak n)\lambda_{M_Kl}(\mathfrak n)}{N(\mathfrak n)^{1/2}}
W_{M_Kl}\left(\frac {(2\pi)^2 N(\mathfrak n)}{|D_K|}\right),  \\
T_1
=& 2 \sum_{X < l \leq 2X}\sum_{\substack{\mathfrak n\subseteq\mathcal O_K, N(\mathfrak n) \leq X^2 \\ \mathfrak n \ne(0) \\ \mathfrak n \neq (n), n \in \mn}}
\frac{\tau_K(\mathfrak n)\lambda_{M_Kl}(\mathfrak n)}{N(\mathfrak n)^{1/2}}
W_{M_Kl}\left(\frac {(2\pi)^2 N(\mathfrak n)}{|D_K|}\right), \\
T_2
= & 2 \sum_{X < l \leq 2X}\sum_{\substack{\mathfrak n \subseteq\mathcal O_K, N(\mathfrak n) > X^2 \\ \mathfrak n \ne(0) \\ \mathfrak n \neq (n), n \in \mn}}
\frac{\tau_K(\mathfrak n)\lambda_{M_Kl}(\mathfrak n)}{N(\mathfrak n)^{1/2}}
W_{M_Kl}\left(\frac {(2\pi)^2 N(\mathfrak n)}{|D_K|}\right).
\end{align*}

   Note again that when $\mathfrak n=(n)$ with $n \in \mn$, we have $\lambda_{M_Kl}(\mathfrak n)=1$ for all $X < l \leq 2X$. It follows that we have
\begin{align*}
  M
= 2\sum_{X < l \leq 2X}\sum_{\substack{n \in \mn}}
\frac{\tau_K(n)}{n}
W_{M_Kl}\left(\frac {(2\pi)^2 n^2}{|D_K|}\right)=\sum_{X < l \leq 2X}\frac 2{2\pi i}\int_{(3)}
\frac{\Gamma(\frac{1}{2}+M_Kl+s)^2}{\Gamma(\frac{1}{2}+M_Kl)^2}
\left(\frac {(2\pi)^2 }{|D_K|}\right)^{-s}e^{s^2}\Big (\sum_{n \geq 1}\frac {\tau_K(n)}{n^{1+2s}}\Big )\frac{ds}{s}.
\end{align*}

  Note that (see \cite[Section 3.8]{iwakow}) a rational prime $p$ ramifies in $O_K$ if and only if $\chi_d(p)=0$, where $\chi_d=\leg {d}{\cdot}$ is the Kronecker symbol. Also, $p$ splits (resp. is inert) in $O_K$ if and only if $\chi_d(p)=1$ (resp. $\chi_d(p)=-1$). We apply this to rewrite the Dirichlet series $\sum_{n \geq 1}\frac {\tau_K(n)}{n^{1+2s}}$ into Euler products to see that we have
\begin{align*}
 \sum_{n \geq 1}\frac {\tau_K(n)}{n^{1+2s}}=\prod_{p \ \ \text{rational prime}}\left(1+\frac {3+\chi_d(p)}{p^{1+2s}}+O(\frac 1{p^{2+4s}}) \right )=\zeta^3(1+2s)L(1+2s, \chi_d)F_K(s),
\end{align*}
  where $L(s, \chi_d)$ is the Dirichlet $L$-function associated to $\chi_d$ and $F_K(s)$ is analytical in the region $\Re(s) \geq -1/4+\varepsilon$ for any $\varepsilon>0$.

   We now shift the line of integration above to $\Re(s)=-1/4+\varepsilon$ to encounter a pole of order $4$ at $s=0$, whose residue can be easily computed to be
\begin{align}
\label{residue1}
 \frac{L(1, \chi_d)F_K(0)\Gamma^{(3)}(\frac{1}{2}+M_Kl)}{2\Gamma(\frac{1}{2}+M_Kl)}+
 O\Big(\frac{\Gamma^{(2)}(\frac{1}{2}+M_Kl)}{2\Gamma(\frac{1}{2}+M_Kl)}+\frac{\Gamma^{'}(\frac{1}{2}+M_Kl)}{2\Gamma(\frac{1}{2}+M_Kl)}+1\Big),
\end{align}
  where the implied constant depends on $K$ only.

  Note that (see \cite[\S 12.31]{WW27}) we have for $\Re(z) \geq 0$,
\begin{align*}
\frac {\Gamma'(z)}{\Gamma(z)}=\log z-\frac 1{2z}-\int^{\infty}_0(\frac 12-\frac 1u+\frac 1{e^u-1})e^{-uz}du.
\end{align*}
  The above implies that
\begin{align*}
\Gamma'(z)=\Gamma(z)\Big (\log z-\frac 1{2z}-\int^{\infty}_0(\frac 12-\frac 1u+\frac 1{e^u-1})e^{-uz}du\Big ).
\end{align*}
  Differentiating both sides above with respect to $z$ now implies that for $\Re(z) \geq 0$,
\begin{align*}
\Gamma^{(n)}(z)=\Gamma(z)\Big ((\log z)^n+O((\log z)^{n-1})\Big ).
\end{align*}
   We apply the above to see that the expression in \eqref{residue1} equals
\begin{align*}
 \frac{L(1, \chi_d)F_K(0)}{2}\log^3 (\frac{1}{2}+M_Kl)+
 O\Big(\log^2 (\frac{1}{2}+M_Kl)\Big).
\end{align*}

  It follows that
\begin{align}
\label{Mest}
  M= \sum_{X < l \leq 2X}\Big (L(1, \chi_d)F_K(0)\log^3 (\frac{1}{2}+M_Kl)+
 O\Big(\log^2 (\frac{1}{2}+M_Kl)\Big)\Big )+M_1,
\end{align}
  where
\begin{align*}
  M_1= \sum_{X < l \leq 2X}\frac 2{2\pi i}\int_{(-1/4+\varepsilon)}
\frac{\Gamma(\frac{1}{2}+M_Kl+s)^2}{\Gamma(\frac{1}{2}+M_Kl)^2}
\left(\frac {(2\pi)^2 }{|D_K|}\right)^{-s}e^{s^2}\zeta^3(1+2s)L(1+2s, \chi_d)F_K(s)\frac{ds}{s}.
\end{align*}

 We recall the convexity bound for $L(s, \chi_d)$ (see \cite[Exercise 3, p. 100]{iwakow}) asserts that
\begin{align*} 
\begin{split}
   L(s, \chi_d) \ll \begin{cases}
   1 \qquad & \Re(s) >1,\\
   (|d|(1+|s|))^{(1-\Re(s))/2+\varepsilon} \qquad & 0< \Re(s) <1,\\
    (|d|(1+|s|))^{1/2-\Re(s)+\varepsilon} \qquad & \Re(s) \leq 0.
\end{cases}
\end{split}
\end{align*}
  Note that the above estimation also holds by replacing $L(s, \chi_d)$ with $\zeta(s)$ and $d$ with $1$.

  We apply the above and \eqref{Gammabound} to see that
\begin{align*}
 M_1 \ll  \sum_{X < l \leq 2X}\frac{\Gamma(\frac{1}{4}+M_Kl+\varepsilon)^2}{\Gamma(\frac{1}{2}+M_Kl)^2} \ll X^{1/2+2\varepsilon},
\end{align*}
   where the last estimation above follows from \eqref{Gammaratiobound}.

    We apply the above into \eqref{Mest} to see that
\begin{align}
\label{Mainsum}
  M= C_KX(\log X)^3+O(X(\log X)^2),
\end{align}
  where $C_K$ is a constant depending on $K$ only.

  Next, we estimate $T_1$ and $T_2$ by keeping the line of integration in $W_{M_Kl}$ to be at $\Re(s)=3$ for $T_2$ while shifting the line of integration
   to $\Re(s)=-1$ for $T_1$. In this process we encounter a simple pole with residue $1$ at $s=0$, whose contribution to $T_1$ is denoted by $T_{1,1}$. We denote $T_{1,2}$ the contribution to $T_1$ on the new line of integration, so that we have
\begin{align}
\label{T1decomp}
  T_1=T_{1,1}+T_{1,2},
\end{align}
   where
\begin{align}
\label{T112exp}
\begin{split}
T_{1,1}=& 2 \sum_{X < l \leq 2X}\sum_{\substack{\mathfrak n\subseteq\mathcal O_K, N(\mathfrak n) \leq X^2 \\ \mathfrak n\ne(0) \\ \mathfrak n \neq (n), n \in \mn}}
\frac{\tau_K(\mathfrak n)\lambda_{M_Kl}(\mathfrak n)}{N(\mathfrak n)^{1/2}}, \\
T_{1,2}
= & 2 \sum_{X < l \leq 2X}\sum_{\substack{\mathfrak n\subseteq\mathcal O_K, N(\mathfrak n) \leq X^2 \\ \mathfrak n\ne(0) \\ \mathfrak n \neq (n), n \in \mn}}
\frac{\tau_K(\mathfrak n) \lambda_{M_Kl}(\mathfrak n)}{N(\mathfrak n)^{1/2}}\cdot \frac{1}{2\pi i}
\int_{(-1)}
\frac{\Gamma(\frac{1}{2}+M_Kl+s)^2}{\Gamma(\frac{1}{2}+M_Kl)^2}
\left(\frac {(2\pi)^2 N(\mathfrak n)}{|D_K|}\right)^{-s}e^{s^2}\frac{ds}{s}.
\end{split}
\end{align}

  We now apply \eqref{lambdaps} to see that
\begin{align}
\label{T11est}
  T_{1,1} \ll  \sum_{\substack{\mathfrak n=(a+b\omega_K), b \neq 0, N(\mathfrak n) \leq X^2 }}
\frac{\tau_K(\mathfrak n)}{|b|}.
\end{align}

  Note that for any $\mathfrak a \mid \mathfrak n$, we have $N(\mathfrak a) \mid N(\mathfrak n)$. Moreover, if $\mathfrak a=(\alpha)$ and $\mathfrak b=(\beta)$ satisfy that $N(\mathfrak a) = N(\mathfrak b)$, then we have $\alpha=\beta u$ with $u \in U_K$. It follows that we have
\begin{align*}
   \tau_K(\mathfrak n) \ll \tau(N(\mathfrak n)),
\end{align*}  
   where $\tau$ is the usual divisor function on $\mz$.  Thus we have
\begin{align}
\label{taupartialsum}
  \sum_{\substack{\mathfrak n=(a+b\omega_K), b \neq 0, N(\mathfrak n) \leq u }}
\frac{\tau_K(\mathfrak n)}{|b|} \ll \sum_{1 \leq  b \leq \sqrt{u}}\frac 1{b}\sum_{1 \leq  a \leq \sqrt{u}}\tau(N(\mathfrak n)).
\end{align}   

  Note that for $\mathfrak n=(a+b\omega_K)$ we have $N(\mathfrak n)=f_K(a;b)$, where
\begin{align*}
  f_K(a;b)=
\begin{cases}
  a^2-db^2, \quad d \equiv 2, 3 \pmod 4, \\
  a^2+ab+b^2, \quad d \equiv 1 \pmod 4.
\end{cases}
\end{align*}   
  Observe that for any fixed $b \in \mz$, $f_K(a;b)$ is an irreducible quadratic polynomial of $a$ with integer coefficients.
   
  We now recall a result of J. McKee \cite[(12)]{McKee95} that for any $x \geq 2$,
 \begin{align}
\label{taupartialsumasp}
  \sum_{\substack{a \leq x}} \tau(f_K(a;b)) =2x\sum_{m \leq x}\frac {\rho(m)}{m}+O(\sum_{m \leq x}\rho(m))+O(x),
\end{align}
   where for any positive integer $m$, $\rho(m)$ denotes the number of solutions to the congruence equation
\begin{align*}
 f_K(n;b) \equiv 0 \pmod m, \quad 0 \leq n < m.
\end{align*}

 Let $f(x, y) = Ax^2+Bxy + Cy^2$ with $A,B, C \in \mz$ and be a positive definite binary quadratic form, such that
 its discriminant $\Delta = B^2 - 4AC < 0$, and $A > 0$. Given
\begin{align*}
 \beta=
 \begin{pmatrix}
  \beta_1 & \beta_2 \\
  \beta_3 & \beta_4
\end{pmatrix} 
  \in \text{SL}_2(\mz),
\end{align*}
  we define $f^{\beta}(x, y) = f(\beta_1 x + \beta_3y, \beta_2x + \beta_4y)$. Then $f^{\beta}$ has the same discriminant as $f$. Forms
related in this way are called equivalent. The form $f$ is called reduced if $|B| \leq A \leq C$,
with $|B|=  0$ if either $A = |B|$ or $A = C$. We recall from \cite[Section 2]{McKee95} that each positive definite binary quadratic form
is equivalent to precisely one reduced form. The number of reduced forms with
discriminant $\Delta$ is finite, and is denoted $H(\Delta)$. The subgroup of $\text{SL}_2(\mz)$ consisting of
those $\beta$ such that $f^{\beta} =f$ is called the automorphism group of $f$.

   Write $(A,B,C)$ as a shorthand for $Ax^2+Bxy + Cy^2$. Let $w(A,B,C)$ denote
the size of the automorphism group of $(A,B,C)$. Then $w(A,A,A) = 6, w(A,0,A) = 4$,
and otherwise, if $(A,B, C)$ is reduced, we have $w(A,B, C) = 2$. Define the Hurwitz class number $H^*(\Delta)$ such that
\begin{align*}
 H^*(\Delta)=\sum_{\substack{(A,B, C) \ \text{reduced} \\ B^2-4AC=\Delta}}\frac {2}{w(A,B,C)}.
\end{align*}

   Note that the discriminant $\Delta_K$ of our function $f_K(n;b)$ equals
\begin{align}
\label{disc}
 \Delta_K=
\begin{cases}
  -4db^2, \quad d \equiv 2, 3 \pmod 4, \\
  -db^2, \quad d \equiv 1 \pmod 4.
\end{cases}
\end{align}    

  It is then shown in \cite[(16)]{McKee95} that for any $x \geq 2$,
\begin{align*}
  \sum_{m \leq x}\rho(m)=\sum_{\substack{(A,B, C) \text{reduced} \\ B^2-4AC=\Delta_K}}\frac {1}{W(A,B,C)}\sum_{g \leq \sqrt{x}}\mu(g)
  \sum_{1\leq Ax^2+Bxy+Cy^2 \leq x/g^2}1,
\end{align*}
  where $\mu$ is the M\"obius function.
  
 Observe that the inner sum is the number of non-trivial integer points within the ellipse $Ax^2+Bxy+Cy^2 = x/g^2$, we diagonalize it to see that its equation can be recasted as
\begin{align}
\label{ellipse}
  \nu_1X^2+\nu_2Y^2=\frac {x}{g^2},
\end{align} 
  where $\nu_{1}\nu_{2}=-\Delta_K/4$. The area of the ellipse given in \eqref{ellipse} then equals
\begin{align*}
  \pi\sqrt{\frac {x}{g^2\nu_1}\cdot \frac {x}{g^2\nu_2}}=\frac {2\pi x}{g^2\sqrt{\Delta_K}}.
\end{align*} 
   It follows that the number of non-trivial integer points within the ellipse $Ax^2+Bxy+Cy^2 = x/g^2$ is 
\begin{align*}
  \ll \frac {x}{g^2\sqrt{\Delta_K}}.
\end{align*}    

  We apply the above and proceed as in the paragraph below \cite[(16)]{McKee95} to see that for any $x \geq 2$,
 \begin{align}
\label{rhopartialsum}
  \sum_{m \leq x}\rho(m) \ll \frac {H^*(\Delta_K)}{\sqrt{-\Delta_K}}x.
\end{align}
  
  It follows from the above and partial summation that we have
\begin{align}
\label{rhoddpartialsum}
  \sum_{m \leq x}\frac {\rho_f(m)}{m} \ll \frac {H^*(\Delta_K)}{\sqrt{-\Delta_K}}\log x.
\end{align}

  We deduce from \eqref{taupartialsum}--\eqref{disc}, \eqref{rhopartialsum} and \eqref{rhoddpartialsum} that for $u \geq 2$, 
\begin{align}
\label{Ttaupartialsumbound}
   \sum_{\substack{\mathfrak n=(a+b\omega_K), b \neq 0, N(\mathfrak n) \leq u }}
\frac{\tau_K(\mathfrak n)}{|b|}  \ll \sum_{1 \leq  b \leq \sqrt{u}}\frac 1{b}\sum_{1 \leq  a \leq \sqrt{u}}\tau(f_K(a;b)) 
  \ll \sqrt{u}\log u\sum_{1 \leq  b \leq \sqrt{u}}\frac {H^*(\Delta_K)}{b^2}.
\end{align}

  Note that we have $H^*(\Delta_K) \ll H(\Delta_K)$.  
We then note from \cite[\S 6, (15)]{Da} that
\begin{align}
\label{Hbound}
  H^*(\Delta_K) \ll H(\Delta_K)  \ll \sqrt{\Delta_K}L(1, \chi_{\Delta_K}) \ll |b|L(1, \chi_{\Delta_K}) \ll |b|,
\end{align}
  where the last estimation above follows by noting that $\chi_{\Delta_K}$ is induced by the character $\chi_{-d}$ so that $L(1, \chi_{\Delta_K}) \ll 1$. 

   We substitute \eqref{Hbound} into \eqref{Ttaupartialsumbound} to see that
\begin{align}
\label{Ttaupartialsumbound1}
   \sum_{\substack{\mathfrak n=(a+b\omega_K), b \neq 0, N(\mathfrak n) \leq u }}
\frac{\tau_K(\mathfrak n)}{|b|} \ll \sqrt{u}(\log u)^2.
\end{align}
  We apply \eqref{Ttaupartialsumbound1} into \eqref{T11est} to see that
\begin{align}
\label{T11bound}
  T_{1,1} \ll & X(\log X)^2.
\end{align}

   We now estimate $T_{1,2}$ by noting that similar to \eqref{sumlambdagamma}, we have for $\mathfrak n=(a+b\omega_K)$ with $a, b \in \mz, b \neq 0$,
\begin{align*}
\begin{split}
 \sum_{X < l\leq 2X}\frac {\lambda_{l}(\mathfrak n)}{\sqrt{N(\mathfrak n)}} \frac{\Gamma(\frac{1}{2}+M_Kl+s)^2}{\Gamma(\frac{1}{2}+M_Kl)^2}\ll
  \frac {|s|}{|b|}X^{2\sigma}.
\end{split}
\end{align*}

   We apply the above to see from \eqref{T112exp} that
\begin{align}
\label{T12exp}
\begin{split}
T_{1,2}
\ll &  X^{-2}\sum_{\substack{\mathfrak n=(a+b\omega_K), b \neq 0, N(\mathfrak n) \leq X^2 }}
\frac {\tau_K(\mathfrak n) N(\mathfrak n)}{|b|} \ll X(\log X)^2,
\end{split}
\end{align}
  where the last estimation above follows from \eqref{Ttaupartialsumbound1} and partial summation.
  
  Note that the above procedure also implies that
\begin{align}
\label{T2exp}
\begin{split}
T_{2}
\ll &  X^{3}\sum_{\substack{\mathfrak n=(a+b\omega_K), b \neq 0, N(\mathfrak n) > X^2 }}
\frac {\tau_K(\mathfrak n) }{|b|N(\mathfrak n)^3} \ll X(\log X)^2. 
\end{split}
\end{align}  
  
   The asymptotical expression given in \eqref{2stmomentasymp} now follows from \eqref{L2dcomp}, \eqref{Mainsum}, \eqref{T1decomp}, \eqref{T11bound}--\eqref{T2exp}. This completes the proof of Theorem \ref{2ndmoment}.

\vspace*{.5cm}

\noindent{\bf Acknowledgments.}  P. G. is supported in part by NSFC grant 12471003.
\bibliography{biblio}

\begin{bibdiv}
\begin{biblist}

\bib{Da}{book}{
      author={Davenport, H.},
       title={{M}ultiplicative {N}umber {T}heory},
     edition={Third edition},
      series={Graduate Texts in Mathematics},
   publisher={Springer-Verlag},
     address={Berlin, etc.},
        date={2000},
      volume={74},
}

\bib{Holm25}{article}{
      author={Holm, K.},
       title={The $1$-level density for zeros of {H}ecke {$L$}-functions of
  imaginary quadratic number fields of class number $1$},
        date={2025},
     journal={J. Th\'eor. Nombres Bordeaux},
      volume={37},
      number={1},
       pages={237\ndash 283},
}

\bib{iwakow}{book}{
      author={Iwaniec, H.},
      author={Kowalski, E.},
       title={{A}nalytic {N}umber {T}heory},
      series={American Mathematical Society Colloquium Publications},
   publisher={American Mathematical Society},
     address={Providence},
        date={2004},
      volume={53},
}

\bib{Marcus18}{book}{
      author={Marcus, D.~A.},
       title={Number fields},
     edition={Second edition},
      series={Universitext},
   publisher={Springer, Cham},
        date={2018},
}

\bib{McKee95}{article}{
      author={McKee, J.},
       title={On the average number of divisors of quadratic polynomials},
        date={1995},
     journal={Math. Proc. Cambridge Philos. Soc.},
      volume={117},
       pages={389–\ndash 392},
}

\bib{MVa1}{book}{
      author={Montgomery, H.~L.},
      author={Vaughan, R.~C.},
       title={{M}ultiplicative number theory. {I}. {C}lassical theory},
      series={Cambridge Studies in Advanced Mathematics},
   publisher={Cambridge University Press},
     address={Cambridge},
        date={2007},
      volume={97},
}

\bib{sound1}{article}{
      author={Soundararajan, K.},
       title={Nonvanishing of quadratic {D}irichlet {$L$}-functions at
  $s=\frac{1}{2}$},
        date={2000},
     journal={Ann. of Math. (2)},
      volume={152},
      number={2},
       pages={447\ndash 488},
}

\bib{Waxman21}{article}{
      author={Waxman, E.},
       title={Lower order terms for the one-level density of a symplectic
  family of {H}ecke {$L$}-functions},
        date={2021},
     journal={J. Number Theory},
      volume={221},
       pages={447\ndash 483},
}

\bib{white22}{article}{
      author={White, D.},
       title={Extreme values of {H}ecke {$L$}-functions to angular characters},
        date={Preprint},
        note={arXiv:2208.02309},
}

\bib{WW27}{book}{
      author={Whittaker, E.~T.},
      author={Watson, G.~N.},
       title={A course of modern analysis},
      series={Cambridge Mathematical Library},
   publisher={Cambridge University Press, Cambridge},
        date={1996},
}

\end{biblist}
\end{bibdiv}
\bibliographystyle{amsxport}

\end{document}